\documentclass[11pt]{amsart} 
\usepackage{amsfonts}
\usepackage{mathrsfs}
\usepackage{geometry}

\usepackage[all]{xypic}

\usepackage{graphicx}
\usepackage{float} 
\usepackage{color}
\usepackage{amssymb}
\usepackage{amsmath}
\usepackage{yhmath}
\usepackage{mathrsfs}
\usepackage[all]{xy}
\usepackage{amsthm}
\usepackage{titletoc}
\usepackage{tikz}
\usepackage[colorlinks,linkcolor=blue,anchorcolor=blue,citecolor=blue]{hyperref}

\def\atemp{dvipspdf}\ifx\atemp\usewhat
\usepackage[dvips,CJKbookmarks=true,bookmarksnumbered=true,bookmarksopen=true,
colorlinks=false,    
pdfborder={0 0 1},   
citecolor=blue,linkcolor=red,anchorcolor=green,urlcolor=blue,
breaklinks=true]{hyperref}
\usepackage{breakurl} 
\fi

\def\atemp{dvipdfmx}\ifx\atemp\usewhat
\usepackage[dvipdfm, 
CJKbookmarks=true,bookmarksnumbered=true,bookmarksopen=true,
colorlinks=false,pdfborder={0 0 1},citecolor=blue,linkcolor=red,
anchorcolor=green,urlcolor=blue,breaklinks=true]{hyperref}
\AtBeginDvi{} 
\fi

\def\atemp{pdflatex}\ifx\atemp\usewhat
\usepackage{cmap}                       
\usepackage[pdftex,CJKbookmarks=true,bookmarksnumbered=true,bookmarksopen=true,
colorlinks=false,pdfborder={0 0 1},citecolor=blue,linkcolor=red,
anchorcolor=green,urlcolor=blue,
breaklinks=true]{hyperref}
\fi

\newtheorem{theorem}{\rm\bf Theorem}[section]
\newtheorem{proposition}[theorem]{\rm\bf Proposition}
\newtheorem{lemma}[theorem]{\rm\bf Lemma}
\newtheorem{corollary}[theorem]{\rm\bf Corollary}
\newtheorem*{theorem 1}{\rm\bf Proposition 1}
\newtheorem*{theorem 2}{\rm\bf Proposition 2}

\theoremstyle{definition}
\newtheorem{definition}[theorem]{\rm\bf Definition}

\theoremstyle{remark}
\newtheorem{remark}[theorem]{\rm\bf Remark}

\newtheorem{conjecture}[theorem]{\rm\bf Conjecture}

\newcommand{\TS}{\mathcal{T}(S)}

\newcommand{\GMT}{\partial_{\mathrm{GM}}\mathcal{T}(S)}

\newcommand{\MF}{\mathcal{MF}}

\newcommand{\PMF}{\mathcal{PMF}}

\newcommand{\T}{\mathcal{T}(S)}
\newcommand{\Ext}{\mathrm{Ext}}

\usepackage{microtype}

\begin{document}
	\title[Optimal geodesics]{Optimal geodesics for boundary points of the Gardiner-Masur compactification}
	
	\author{Xiaoke Lou}
	\author{Weixu Su}
	\author{Dong Tan}
	
	\address{Xiaoke Lou: School of Mathematical Sciences, Fudan University, Shanghai 200433, China}
	\email{xklou20@fudan.edu.cn}
	
	\address{Weixu Su: School of Mathematics, Sun Yat-sen University, Guangzhou 510275, China}
	\email{suwx9@mail.sysu.edu.cn}
	
	\address{Dong Tan: School of Mathematics and Information Science, Guangxi University, Guangxi, 530004, China.}
	\email{duzuizhe2013@foxmail.com}

	\begin{abstract}
		The Gardiner-Masur compactification of Teichm\"uller space is homeomorphic to the horofunction compactification of the Teichm\"uller metric.
		Let $\xi$ and $\eta$ be a pair of boundary points in the Gardiner-Masur compactification that fill up the surface. 
		We show that there is a unique  Teichm\"uller geodesic which is optimal for the horofunctions corresponding to $\xi$ and $\eta$.
		In particular, when $\xi$ and $\eta$ are Busemann points that fill up the surface, the geodesic converges to $\xi$ in forward  direction and to $\eta$ in backward direction.
		As an application, we show that if $\mathbf{G}_n$ is a sequence of Teichm\"uller geodesics passing through $X_n$ and $Y_n$ such that 
		$X_n \to \xi$ and $Y_n \to \eta$, 
		then  $\mathbf{G}_n$ converges to a unique Teichm\"uller geodesic.
		
	\end{abstract}
	\maketitle
	
	\section{Introduction}
	
	\noindent
	
	\subsection{Background}	
	Let $S=S_{g,n}$ be an oriented closed surface of genus $g$ with $n$ punctures, where $3g-3+n \geq 2$.  
	Let $\TS$ be the Teichm\"uller space of $S$.
	It is naturally a complex manifold  of dimension $3g-3+n$. 
	
	We shall endow $\TS$ with the Teichm\"uller metric, which is equal to the Kobayashi metric  \cite{Royden}. 
	It is known that the Teichm\"uller space of torus is isometric to the hyperbolic plane. In higher dimensions,  the situation is very different:
	\begin{enumerate}
		\item Royden \cite{Royden} proved that the 
		Teichm\"uller space endowed with the Teichm\"uller metric is not homogeneous and there is no point about which it is symmetric. Moreover, Royden's theorem  (extended by Earle-Kra \cite{EK}) about automorphisms of Teichm\"uller space states that any isometry of the Teichm\"uller metric is induced by an element of the mapping class group. 
		
		\item Masur \cite{Masur1975On} showed that the Teichm\"uller metric does not have negative curvature in the sense of Busemann; 
		Masur and Wolf \cite{Masur1994Teichmuller} showed that the Teichm\"uller metric is not Gromov hyperbolic. See also \cite{Minsky}.  
		
		\item The Teichm\"uller space is biholomorphic to a bounded pseudoconvex domain \cite{BE}.  Recently, Markovic \cite{markovic} showed that
		the Teichm\"uller space  is not
		biholomorphic to a bounded convex domain in $\mathbb{C}^{3g-3+n}$.
	\end{enumerate}

	Nevertheless, there is a strong analogy between the geometry of Teichm\"uller space and that of a hyperbolic space
	\cite{ABEM, Bers, MM1999, Rafi}.
	The aim of this paper is to explore some  properties of the horofunction boundary of $\TS$, inspired by analogies with
	negatively curved spaces.

	\subsection{Gardiner-Masur boundary}	
	Horofunction compactification and the corresponding boundary of a proper metric space were introduced by Gromov \cite{Gromov81}. These notions have been recently studied in several research
	areas, including geometry of bounded convex domains with the Hilbert metric, Teichm\"uller spaces, and random walks on groups of isometries \cite{Walsh, Gouezel}.

	Liu-Su \cite{liu2010horofunction} showed that the horofunction compactification of $\mathcal{T}(S)$ equipped with the Teichm\"uller metric is homeomorphic to the Gardiner-Masur compactification. 	
	
	Unlike the Thurston boundary, which is homeomorphic to a sphere of dimension $6g-7+2n$, 
	the topology of the Gardiner-Masur boundary is more complicated.  Note that every Teichm\"uller geodesic ray converges to a unique limit on
	the Gardiner-Masur boundary. Limits of geodesic rays are called \textit{Busemann points}.  
	
	Busemann points in  the Gardiner-Masur boundary are completely characterized by Walsh \cite{Walsh2019The}.
	There exist non-Busemann points in the Gardiner-Masur boundary \cite{miyachi2014extremal}.
	Recently, Azemar \cite{Azemar2021qualitative} showed that
	Busemann points are not dense in the Gardiner-Masur boundary.

	\bigskip

	In a proper $\delta$-hyperbolic metric space,  
	any interior point and any
	boundary point  can be connected by a geodesic ray. 
	Moreover, any two distinct boundary points
	can be connected by a (bi-infinite) geodesic.
	See  \cite{KM}.
	
	By identifying the Gardiner-Masur compactification with the horofunction compactification of $\TS$,
	each  boundary point
	is corresponding to a horofunction. 
	Azemar \cite{Azemar2021qualitative} showed that, 
	for any given horofunction $\Psi$ and 
	any given $X\in \TS$, there is a unique \textit{optimal geodesic} for 
	$\xi$ passing through $X$. 
	By definition, a  geodesic $\mathbf{G}(t)$ is  optimal for $\Psi$ if
	for every $t\in \mathbb{R}$, 
	$$
	\Psi(\mathbf{G}(t))-\Psi(\mathbf{G}(0))=-t. 
	$$
	Equivalently, the optimal geodsic is a  gradient line along which the horofunction decreases fastest. 
	
	When $\Psi$ is a Busemann function (horofunction function corresponding to Busemann point),   $\mathbf{G}(t)$  is exactly the geodesic that converges to $\Psi$ in forward direction.

	\subsection{Main results}
	
	The main result in this paper is to determine when
	two distinct boundary points of $\TS$ are ``connected" by a Teichm\"uller geodesic.

	Let $\mathcal{MF}$  be the space of measured foliations on $S$.
	Every point $\xi$ in the Gardiner-Masur boundary is
	determined by a nonnegative homogeneous function
	$i(\xi, \cdot) \colon  \mathcal{MF}\to \mathbb{R}_{\geq 0}.$
	The representation is unique up to multiplication by a positive constant \cite[Theorem 1.1]{Miyachi2008Teichm}.


	The following definition is a generalization  of a pair of measured foliations that fill up the surface.
	
	\begin{definition}	Let $\xi$ and $\eta$ be a pair of points in the Gardiner-Masur boundary. We say that $\xi$ and $\eta$
		{fill up the surface} if
		$i(\xi, \mu) + i(\eta,\mu) > 0$
		for all $\mu \in \mathcal{MF}\setminus\{0\}.$
	\end{definition}

	Let $\xi$ and $\eta $ be  a pair of points in the Gardiner-Masur boundary.
	A Teichm\"uller geodesic $\mathbf{G}(t)$ is said to be\textit{ optimal} for $\xi$ and 
	$\eta$ if 
	$\mathbf{G}(t)$ is an optimal geodesic for $\xi$ and
	the reverse geodesic $\widetilde{\mathbf{G}}(t)=\mathbf{G}(-t)$ is an optimal geodesic for $\eta$.
	We prove
	
	\begin{theorem}\label{thm:horofunction}
		Let $\xi$ and $\eta $ be  a pair of points in the Gardiner-Masur boundary that fill up the surface. There is a unique Teichm\"uller geodesic that is optimal for $\xi$ and $\eta$.
	\end{theorem}
	
	Theorem \ref{thm:horofunction} does not indicate that the geodesic converges to $\xi$ or $\eta$, in case that $\xi$ and $\eta$ are not Busemann points. 
	
	A special case of Theorem \ref{thm:horofunction} is
	
	\begin{theorem}\label{thm:Busemann} Let $\xi,\eta$ be a pair of Busemann points
		in the Gardiner-Masur boundary that fill up the surface. There exists a unique Teichm\"uller geodesic $\mathbf{G}(t)$ such that
		$$
		\lim _{t \rightarrow +\infty} \mathbf{G}(t)=\xi \quad \text { and } \quad \lim _{t \rightarrow-\infty} \mathbf{G}(t)=\eta.
		$$	
		
	\end{theorem}

		The assumption that $\xi$ and $\eta$
	fill up the surface is necessary both in Theorem \ref{thm:horofunction} and  
	Theorem \ref{thm:Busemann}.

	\begin{remark}
		Gardiner-Masur \cite[Theorem 5.1]{gardiner1991extremal} proved that,
		given any pair of measured foliations $\mu,\nu$ that fill up the surface, there is a unique
		Teichm\"uller geodesic determined by a holomorphic quadratic differetnial whose horizontal and vertical
		foliations are in the projective classes of $\mu$ and $\nu$. 
		However, the geodesic may not converges to the  $\mu$ or $\nu$ in general. 
	\end{remark}

	\subsection{Idea of the proofs}
	
	Theorem \ref{thm:Busemann} can be proved directly. 
	The set of Busemann points  in the  
	Gardiner-Masur boundary is completely characterized by Walsh \cite[Theorem 8]{Walsh2019The}.
	Given a Teichm\"uller geodesic ray, the limited Busemann point 
	can be  represented explicitly in terms of the decomposition of initial quadratic differential. 
	The problem of  existence (and uniqueness) of
	geodesic in  Theorem \ref{thm:Busemann} is reduced to the existence (and uniqueness) of {positive} eigenvector of a Perron-Frobenius matrix.

	In general,  there is no explicit representation for a  horofunction,
	thus the previous argument for Theorem \ref{thm:Busemann} doesn't work. 
	Instead, we consider level sets of {the horofunction} $\xi$, defined by 
	$$\mathrm{L}(\xi, s)=\{X\in \TS \ : \Psi_\xi(X)=s\},$$
	where $\Psi_\xi$ is the horofunction corresponding to $\xi$. 
	
	The filling condition of $\xi$ and $\eta$ implies that each $\mathrm{L}(\xi, s)$
	is tangent to some $\mathrm{L}(\eta, t)$. We observe that, if we take $X$ to be an intersection point of  $\mathrm{L}(\xi, s)$ and $\mathrm{L}(\eta, t)$, then the optimal geodesic 
	for $\xi$ passing through  $X$, by changing its direction,  is the  optimal geodesic for $\eta$ {passing through $X$}.
	
	The key proof is based on the following inequality.
	Associated with a Teichm\"uller geodesic, there
	are two natural horofunctions $\Psi_{{F}}$ and $B_\xi$, where $\Psi_{{F}}$ is corresponding to the vertical foliation of the quadratic differential determining the geodesic, and $B_\xi$ is the limited Busemann function.
	Then we have
	
	\begin{equation}
		\Psi_{F}\leq \Psi_\xi\leq B_{\xi}.
	\end{equation}
	
	Note that the level sets of $\Psi_{F}$ are corresponding to level sets of  the extremal length function  $\Ext_{F}$, which was studied in \cite{Su2018Horospheres}. In this paper, we will give a proof that the Busemann function $B_{\xi}$ is $C^1$. 
	Using the above inequality, we are able to show that for each $s$,  the level set $\mathrm{L}(\xi, s)$ intersects with $\mathrm{L}(\eta, t)$ at a unique point $X_s$,
	where $t$ depends on $s$. Consequently, the loci of intersection points
	$\{X_s\}$  define the optimal geodesic.

	\subsection{An application}

	We apply Theorem \ref{thm:horofunction} to show 
	\begin{theorem}\label{thm:convergence}
		Let  $\xi$ and $\eta$ be a pair of Gardiner-Masur boundary points that  fill up the surface. 
		For any sequences $X_n,Y_n\in\TS$  such that $X_n\to\xi$ and $Y_n\to\eta$, 
		let $\mathbf{G}_n$ be the Teichm\"uller geodesics passing through $X_n,Y_n\in\TS$.
		Then $\mathbf{G}_n$ converge to a unique geodesic as $n\to \infty$.  The limited geodesic is exactly  the optimal geodesic in  Theorem \ref{thm:horofunction}.
	\end{theorem}


	To prove Theorem \ref{thm:convergence} , we first show that there is a compact set $\mathcal{K}\subset \TS$ such that each geodesic $\mathbf{G}_k$  intersects $\mathcal{K}$. 
	This is confirmed by the assumption that $\xi$ and $\eta$ fill up the surface.
	Up to a subsequence, $\mathbf{G}_k$ converges to some geodeic $\mathbf{G}$. 
	Finally, we show that $\mathbf{G}$ is an optimal geodesic for $\xi$ and $\eta$. 
	Then Theorem \ref{thm:convergence} follows from the uniqueness of optimal geodesic.

	A corollary of Theorem \ref{thm:convergence} 
	is that Teichm\"uller geodesics are ``sticky".
	See Theorem \ref{thm:sticky}.
	The notion is introduced by Duchin and Fisher \cite{DF}.
	They considered the Thurston compactification of $\TS$ and made the following conjecture.
	
	\begin{conjecture}[Duchin-Fisher] Let $\mathbf{G}$ be a Teichm\"uller geodesic ending at two simple closed curves $\alpha$ and $\beta$ (it is necessary that $\alpha$ and $\beta$ fill up the surface). There exists a compact set $\mathcal{K}\subset \TS$ such that for any sequences $X_n \to \alpha$ and $Y_n \to \beta$ in the Thurston compactification, the geodesic segments connecting $X_n$ and $Y_n$ intersect $\mathcal{K}$ for sufficiently large $n$.
	\end{conjecture}
	
	Theorem \ref{thm:convergence} shows that the above conjecture is true if we replace the Thurston compactification by the Gardiner-Masur compactification. 
	However, we don't know whether $X_n \to \alpha$ (where $\alpha$ is a simple closed curve) in the Thurston compactification implies $X_n \to \alpha$ in the Gardiner-Masur compactification or not. 
	The two convergences are equivalent if the limit is a uniquely ergodic measured foliation \cite[Corollary 1]{miyachi2013teichmuller}.

	\subsection{Organization of the paper}
	In \S 2 we give some preliminaries on Teichm\"uller theory and measured foliations. 
	In \S 3 we summarize without proofs the relevant results on
	the Gardiner-Masur compactification of $\TS$.
	In \S 4, we prove Theorem \ref{thm:Busemann}. We study the geometry of {level sets} in \S 5 and prove Theorem \ref{thm:horofunction}.
	In \S 6, we prove Theorem \ref{thm:convergence} .

	\section{Preliminaries}\label{sec:pre}
	
	\noindent 
	
	In this section, we give a brief exposition of Teichm\"uller theory of Riemann surfaces, quadratic differentials and measured foliations.
	More details can be found in the books \cite{FM2012,fathi1979travaux}.
	
	\subsection{Teichm\"uller space}
	
	Let $S$ be an oriented closed surface of genus $g$ with $n$ punctures,
	where $3g - 3 +n \geq 2$. 
	The \emph{Teichm\"uller space} $\mathcal{T}(S)$ is the space of equivalence classes of marked Riemann surfaces
	$(X, f)$, where $X$ is a Riemann surface of  analytically finite type
	$(g,n)$ and  $f : S \to X$
	is an orientation-preserving homeomorphism. The equivalence relation is given by {$(X_1, f_1) \sim (X_2, f_2)$} if there is a conformal mapping
	$h : X_1 \rightarrow X_2$ homotopic to  $f_2 \circ f_1^{-1}$.

	For any two equivalence classes
	$[(X_1, f_1)], [(X_2,f_2)] \in \mathcal{T}(S)$,
	their \emph{Teichm\"uller distance} is defined by
	$$ d_{\mathcal{T}}([(X_1, f_1)], [(X_2,f_2)]) = \frac{1}{2} \inf_{\phi} \log K(\phi),$$
	where $\phi$ is taken over all quasiconformal mappings
	$\phi: X_1 \rightarrow X_2$
	homotopic to $f_2\circ f_1^{-1}$, and $K(\phi)$ is the maximal quasiconformal
	dilatation of $\phi$.
	
	For  simplicity of notation, we shall denote a point in $\mathcal{T}(S)$ by a Riemann surface $X$,
	without explicit reference to the marking or to the equivalence relation.
	
	The Teichm\"uller metric is a complete Finsler metric on $\TS$. Any two distinct points in $\TS$ are connected by a unique geodesic. 
	Teichm\"uller geodesic will be defined in \S \ref{subsec:QD}.

	\subsection{Measured foliations}
	Recall that a \textit{measured foliation} $F$ on $S$ is a foliation (with a finite number of singularities) with a transverse invariant measure. This means that if the local coordinates send the regular leaves of $F$ to horizontal $\operatorname{arcs}$ in $\mathbb{R}^2$, then the transition functions on $\mathbb{R}^2$ are of the form $(f(x, y), \pm y + c$) where $c$ is a constant, and the measure is given by $|d y|$. The allowed singularities of $\mathcal{F}$ are topologically the same as those that occur at $z=0$ in the line field of $z^{p-2} d z^2, p \geq 3$ or  $p=1$ at the puncture of $S$.
	
	Let $\mathcal{S}$ be the set of free homotopy classes of essential (non-trivial, non-peripheral) simple closed curves on $S$. The intersection number $i(\gamma, F)$ of a simple closed curve $\gamma$ with a measured foliation $F$ endowed with transverse measure $\mu$ is defined by
	$$
	i(\gamma, F)=\inf _{\gamma^{\prime}} \int_{\gamma^{\prime}} d \mu,
	$$
	where the infimum is taken over all simple closed curves $\gamma^{\prime}$ in the isotopy class of $\gamma$.
	Two measured foliations $F$ and $F^{\prime}$ are measure equivalent if, for all $\gamma \in \mathcal{S}$, $$i(\gamma, F)=i\left(\gamma, F^{\prime}\right).$$ Denote by $\mathcal{M F}=\mathcal{M F}(S)$ the space of equivalence classes of measured foliations on $S$.
	
	Two measured foliations ${F}$ and ${F}^{\prime}$ are projectively equivalent if there is a constant $b>0$ such that ${F}=b \cdot {F}^{\prime}$, i.e. $i(\gamma, {F})=b \cdot i\left(\gamma, {F}^{\prime}\right)$ for all $\gamma \in \mathcal{S}$. The space of projective equivalence classes of foliations is denoted by $\mathcal{P M F}$.
	
	Thurston showed that $\mathcal{M F}$ is homeomorphic to a $6 g-6+2 n$ dimensional ball and $\mathcal{P M \mathcal { F }}$ is homeomorphic to a $6 g-7+2 n$ dimensional sphere,
	and the set $\mathcal{S}$ (corresponding to
	one-cylinder measured foliations) is dense in $\mathcal{PMF}$. We refer to \cite{fathi1979travaux}
	for more details.

	The space $\mathcal{PMF}$ can be considered as a completion of  $\mathcal{S}$. 
	It plays an important role in the work of Thurston on classification of mapping classes.
	Following Thurston, 
	$\mathcal{PMF}$ is a natural boundary of $\TS$ called the Thurston boundary. 
	In \S \ref{sec:GM}, we will introduce a larger boundary of $\TS$ called the 
	Gardiner-Masur boundary. 
	
	\bigskip
	
	Given a conformal metric $\rho$ on a Riemann surface $X$, denote the $\rho$-area of $X$ by
	$
	\operatorname{Area}_{\rho}(X)$.
	The extremal length of $\gamma\in \mathcal{S}$ on $X$ is defined by
	$$
	\operatorname{Ext}_{X}(\gamma)=\sup _{\rho} \frac{\left({\inf_{\gamma^{\prime}\simeq \gamma} \int_{\gamma^{\prime}} \rho(z)|\mathrm{dz}|}\right)^{2}}{\operatorname{Area}_{\rho}(X)},
	$$
	where $\rho(z)|\mathrm{d} z|$ ranges over all conformal metrics on $X$ with
	$
	0<\operatorname{Area}_{\rho}(X)<\infty .
	$

	Kerckhoff \cite{kerckhoff1980asymptotic} showed that the notion 
	of extremal length extends continuously to the space of measured foliations. Moreover, he proved the following important distance formula for the Teichm\"uller metric.
	
	\begin{theorem}
		The Teichm\"uller distance between any $X, Y\in \T$ is equal to 
		\begin{eqnarray*}
			d_{\mathcal{T}}(X,Y) &=& \frac{1}{2} \log\sup_{\gamma\in \mathcal{S}}\frac{\operatorname{Ext}_{X}(\gamma)}{\operatorname{Ext}_{Y}(\gamma)} \\
			&=& \frac{1}{2} \log\sup_{\mu\in \mathcal{MF}} \frac{\operatorname{Ext}_{X}(\mu)}{\operatorname{Ext}_{Y}(\mu)}. 
		\end{eqnarray*}
	\end{theorem}

	\subsection{Quadratic differentials and Teichm\"uller geodesics}
	\label{subsec:QD}
	A \emph{holomorphic quadratic differential} $q$ on $X \in \mathcal{T}_{g, n}$ is a tensor which is locally represented by $q=q(z) d z^2$, where $q(z)$ is a holomorphic function on the local conformal coordinate $z$ of $X$. We allow holomorphic quadratic differentials to have at most simple poles at the punctures of $X$. Denote the vector space of holomorphic quadratic differentials on $X$ by $Q(X)$.
	
	For $q\in Q(X)$, its area is defined by 
	$$\|q\|=\int_X |q(z)| \ |d z|^2.$$
	We denote by $Q^1(X)$ the subset of quadratic differentials on
	$X$ of unit area. 
	
	A quadratic differential $q$ determines a pair of transverse measured foliations ${F}_h(q)$ and ${F}_v(q)$, called the \emph{horizontal and vertical foliations} for $q$. Away from the zeroes of 
	$q$, we can {choose} natural coordinates  $z=x+i y$ such that 
	$q=dz^2$. Then the leaves of ${F}_h(q)$ are given by
	$
	y=\text{const;}
	$
	and the leaves of ${F}_v(q)$ are given by
	$
	x=\text{const},
	$
	and the transverse measures are $|d y|$ and $|d x|$. The foliations ${F}_h(q)$ and ${F}_v(q)$ have the zero set of $q$ as their common singular set, and at each zero of order $k$ they have a $(k+2)$-pronged singularity, locally modelled on the singularity at the origin of $z^k d z^2$.

	Teichm\"uller geodesic rays are parameterized by quadratic differentials. Given a Riemann surface $X\in \T$ and a quadratic differential $q \in Q^1(X)$, there is a family of Riemann surfaces $X_t$ (where $t\geq 0$)  such that
	$$
	d_{\mathcal{T}}\left(X, X_t\right)=t .
	$$
	Moreover, the Teichm\"uller maps (i.e., extremal quasiconformal map) $$f_t: X \to X_t$$
	expands along the leaves of the horizontal foliation  by $e^t$ and contracts along the leaves of the vertical foliation by $e^{-t}$. 
	The family $\{X_t\}_{t \geq 0}$ is called the Teichm\"uller geodesic ray determined by $q$,  denoted by $\mathbf{R}_q(t)$.	We can extend 
	$\mathbf{R}_q(t)$ to a complete geodesic $\mathbf{G}_q(t)$ by allowing $t<0$. 
	
	\medskip

	A pair ${F}_{+}, {F}_{-}$ of measured foliations is called \textit{fill up the surface} if for any $\mu\in\MF \setminus\{0\}$,
	$$
	i\left({F}_{+}, \mu\right)+i\left({F}_{-}, \mu\right) \neq 0.
	$$
	
	It is straightforward that the vertical and horizontal foliations of a quadratic differential fill up the surface. The converse is a theorem of Gardiner-Masur \cite{gardiner1991extremal}:

	\begin{theorem}\label{thm:geodesic} If a pair  ${F}_{+}, {F}_{-}$ of measured foliations on $S$ fill up the surface, then there is a Riemann surface $X$ and a holomorphic quadratic differential $q$ on $X$ such that 
		$$ {F}_{+}={F}_{v}(q),{F}_{-}={F}_h(q).$$
		Moreover, both $X$ and $q$ are unique. 
	\end{theorem}
	
	When ${F}_{+}, {F}_{-}$ fill up the surface, they 
	determine a unique Teichm\"uller geodesic
	$\mathbf{G}_q(t)$, where for  each $t$, $\mathbf{G}_q(t)$ is corresponding 
	the pair of measured foliations $e^t {F}_{+}, e^{-t}{F}_{-}$.

	\section{Gardiner-Masur compactification}\label{sec:GM}
	
	In this section, we fix a base point $X_{0}$ in $\mathcal{T}(S)$.
	Let $\mathbb{R}_{+}^{\mathcal{S}}$ be the space of nonnegative functions
	on $\mathcal{S}$, endowed with the weak topology.
	Let $\mathrm{P}\mathbb{R}_{+}^{\mathcal{S}}$ be the set of projective classes of
	{$\mathbb{R}_{+}^{\mathcal{S}}\setminus\{0\}$}.

	\subsection{Gardiner-Masur boundary}
	We define the embedding $\Phi: \mathcal{T}(S) \to\mathrm{P}\mathbb{R}_{+}^{\mathcal{S}}$ by
	taking
	$$\Phi(X)=
	[(\sqrt{\operatorname{Ext}_{X}(\alpha)})_{\alpha\in \mathcal{S}}].$$
	Gardiner and Masur \cite{gardiner1991extremal} showed that the
	closure of $\Phi(\mathcal{T}(S))$ in $\mathrm{P}\mathbb{R}_{+}^{\mathcal{S}}$
	is compact. The boundary of $\Phi(\mathcal{T}(S))$ is called
	the \emph{Gardiner-Masur boundary}, denoted by $\partial_{\mathrm{GM}}\mathcal{T}(S)$.
	This boundary properly contains $\mathcal{PMF}$.

	Miyachi \cite[Theorem 1.1]{Miyachi2008Teichm} proved

	\begin{theorem}\label{Miyachi2008Teichm}
		Any point $\xi \in \partial_{\mathrm{GM}}\mathcal{T}(S)$ can be represented as
		a nonnegative continuous function $	i(\xi, \mu), \  \mu\in \MF$ on
such that
	 \begin{equation}\label{equ:im1}
	 	i(\xi, r \cdot \mu) = r \cdot  i(\xi,\mu)
	 \end{equation}
	 for all $r > 0$ and $\mu \in \mathcal{MF}$.
			The assignment $\mathcal{S} \ni \alpha
			\mapsto i(\xi, \alpha)$
			determines $\xi$ as a point of $\partial_{\mathrm{GM}}\mathcal{T}(S)$.
		Furthermore, the function $i(\xi, \cdot)$ is unique up to
		multiplication by a positive constant.
	\end{theorem}
	
	We don't have an explicit formula for a general
	$\xi$. The following inequality is useful.

\begin{proposition}\label{prop_low-bound}
		Let $X_n \in \mathcal{T}(S)$ converge to a boundary point $\xi \in \partial_{\mathrm{GM}}\mathcal{T}(S)$. Assume that $q_n \in {Q^1(X_0)}$ is the initial quadratic differential associated to the Teichm\"uller map from $X_0$ to $X_n$. Up to a subsequence, we also assume that the vertical measured foliations ${F}_{v}(q_n)$ of $q_n$ converge to a measured foliation
		${F}$. Then we have
		$$
		i({F},\mu) \leq i(\xi,\mu)
		$$
		for any $\mu \in \mathcal{MF}$.
	\end{proposition}
In fact, the limit of ${F}_{v}(q_n)$ 
is unique, see \cite[Theorem 1.1]{Azemar2021qualitative} and \cite[Theorem 3.4]{LiuShi}. Proposition \ref{prop_low-bound} can be proved by  using Minsky's inequality.
See for instance Miyachi \cite[Proposition 5.1]{Miyachi2008Teichm}, Liu and Shi \cite[Lemma 3.10]{LiuShi}.
We also provide a proof, see Lemma \ref{lem:app1} in the Appendix. 
	
	\begin{remark}\label{remark:one}
		Note that the function $i(\xi, \cdot)$ is unique up to
		multiplication by a positive constant. In Proposition \ref{prop_low-bound},
		$i(\xi, \cdot)$ is chosen as following.
		For each $X_n$, we associate with a function on $\mathcal{MF}$:
	\begin{equation}\label{equ:im2}
		i(X_n, \mu)=\sqrt{\frac{\mathrm{Ext}_{X_n}(\mu)}{K(X_0,X_n)}},
		\end{equation}
		where $K(X_0,X_n)$ is the quasiconformal dilatation of the Teichm\"uller map from
		$X_0$ to $X_n$. Then $i(X_n, \mu)\to i(\xi,\mu)$
		uniformly on compact subsets of $\mathcal{MF}$
		as $X_n\to \xi$.
	\end{remark}

	Following Miyachi \cite{Miyachi2014}, 
	for any $X, Y \in \TS$,  we set
	$$\langle X   |   Y \rangle=\frac{1}{2}\left(	d_{\mathcal{T}}(X_0, X)+ 	d_{\mathcal{T}}(X_0, Y)-	d_{\mathcal{T}}(X,Y) \right).$$
	The intersection number between
	$X$ and $Y$  is defined  by
	
\begin{equation} \label{equ:im3}
	i(X,Y)= e^{-2\langle X    |   Y \rangle}.
\end{equation}
	
		The next theorem, proved by Miyachi \cite[Theorem 4]{Miyachi2014}, unifies the above notions
		\eqref{equ:im1} \eqref{equ:im2} and \eqref{equ:im3}. 
		
	\begin{theorem} \label{thm:Miyachi}
		The intersection number $i(X,Y)$ extends continuously to the Gardiner-Masur compactification.
	 In particular, for any two sequences $X_n$ and $Y_n$ in $\TS$ such that $X_n \to \xi\in \GMT$ and $Y_n\to \eta \in \GMT$, we have 
			$i(X_n, Y_n) \to i(\xi,\eta).$
			\end{theorem}

	Im particular, we consider $\PMF$ as a subset of the Gardiner-Masur boundary. Then each element in $\PMF$ 
	is represented by a measured foliations $\mu$ such that 
	$\mathrm{Ext}_{X_0}(\mu)=1$. For any $\mu, \nu\in \PMF$,
	$i(\mu, \nu)$ is exactly their geometric intersection number.

	\subsection{Horofunction compactification}
	Fix a base point $X_{0} \in \mathcal{T}(S)$. To each $Z \in \mathcal{T}(S)$ we assigned a function $\Psi_{Z}: \mathcal{T}(S) \rightarrow \mathbb{R}$, given by
	$$
	\Psi_{Z}(X)=d_{\mathcal{T}}(X, Z)-d_{\mathcal{T}}\left(X_{0}, Z\right).
	$$
	Let $C(\mathcal{T}(S))$ be the space of continuous functions on $\mathcal{T}(S)$ endowed with the topology of locally uniformly convergence on $\mathcal{T}(S)$. Then the map
	\begin{eqnarray*}
		\Psi: \mathcal{T}(S) &\rightarrow& C(\mathcal{T}(S)) \\
		Z &\mapsto& \Psi_{Z}(\cdot)
	\end{eqnarray*}
	is an embedding from $\mathcal{T}(S)$ into $C(\mathcal{T}(S))$. The closure $\overline{\Psi(\mathcal{T}(S))}$ of $\Psi(\mathcal{T}(S))$ is compact, which is called the\textit{ horofunction compatification} of $\mathcal{T}(S)$ with Teichm\"uller metric. The \textit{horofunction boundary} is defined to be
	$$
	\partial_{h} \mathcal{T}(S)=\overline{\Psi(\mathcal{T}(S))}-\Psi(\mathcal{T}(S))
	$$
	and its elements are called \textit{horofunctions}.
	
	\begin{remark}
		Note that here the definition of $\partial_{h} \mathcal{T}(S)$ depends on the choice of the base point $X_{0}$. If one changes to an alternative base point $X_{1}$, then the assignment of the new function $\Psi_{Z}^{\prime}$ is related to $\Psi_{Z}$ by $\Psi_{Z}^{\prime}(\cdot)=\Psi_{Z}(\cdot)-\Psi_{Z}\left(X_{1}\right)$. Thus there is a natural identification between $\overline{\Psi(\mathcal{T}(S))}$ and $\overline{\Psi^{\prime}(\mathcal{T}(S))}$,
		and elements in  $\partial_{h} \mathcal{T}(S)$ are defined up to additive constants. We refer to \cite{liu2010horofunction, Walsh2019The} for more detail about horofunction boundary.
	\end{remark}

	Liu and Su \cite{liu2010horofunction} proved that:
	\begin{theorem}\label{thm:LS}
		The Gardiner-Masur compactification of $\TS$ is homeomorphic to the horofunction compactification of  $\left(\mathcal{T}(S), d_{\mathcal{T}} \right)$. For each $\xi\in \partial_{\mathrm{GM}} \T$, 
		the corresponding horofunction is represented  by
		$$
		\Psi_\xi(X):= \log \sup _{\mu \in \mathcal{MF}} \frac{i(\xi, \mu)}{\sqrt{\operatorname{Ext}_{X}(\mu)}}
		-
		\log \sup _{\mu \in \mathcal{MF}} \frac{i(\xi, \mu)}{\sqrt{\operatorname{Ext}_{X_0}(\mu)}}.
		$$
	\end{theorem}

	\section{A pair of Busemann points that fill up the surface}
	
	In this section, we give a direct proof of Theorem \ref{thm:Busemann},
	independent of Theorem \ref{thm:horofunction}.

	As an immediate corollary of Theorem \ref{thm:LS},
	every Teichm\"uller geodesic ray has a unique limit on the Gardiner-Masur boundary. 
	The limit point of a geodesic ray is called
	a \emph{Busemann point}.

	Let $\mathbf{R}_{q}: [0, \infty) \rightarrow \mathcal{T}(S)$ be
	the Teichm\"uller geodesic ray determined by $q \in Q^1(X)$.
	Let  $F_v, F_h\in \mathcal{MF}$ be  the vertical and horizontal foliations of $q$.
	Denote the limit of $\mathbf{R}_{q}$
	in $\partial_{\mathrm{GM}}\mathcal{T}(S)$ by $\xi$.

	With the above notations, we have (see Walsh \cite[Theorem 1]{Walsh2019The}) 
	\begin{theorem}\label{Walsh2012Corollary1.2} Let $\sum\limits_{i=1}^{k} F_i$ be the ergodic decomposition
		of $F_v$. Then the Busemann point is represented by 
		
		$$ i(\xi,\mu) =
		\Big\{ \sum_{i=1}^{k}\frac{i({F}_{i},\mu)^{2}}
		{i({F}_{i}, {F}_{h})}
		\Big\}^{1/2}, \  \mu\in\MF.$$
	\end{theorem}
	
	\bigskip
	
	\begin{proof}[Proof of Theorem \ref{thm:Busemann}]
		Let $\xi, \eta \in \partial_{\mathrm{GM}}\mathcal{T}(S)$ be a pair of Busemann points that fill up the surface.
		
		First, we observe that any possible geodesic need to satisfy a system of equations.
		Assume that a pair of measured foliations ${F}_{+}, {F}_{-}$ fill up the surface, and they	determine a Teichm\"uller geodesic  $\mathbf{G}_{q}$
		with 
		$$\lim_{t\to +\infty} \mathbf{G}_{q}(t) =\xi, \lim_{t\to -\infty} \mathbf{G}_{q}(t) =\eta.$$

		Using Theorem \ref{Walsh2012Corollary1.2}, we can write 
		\begin{equation*}
			i(\xi, \cdot) =
			\sqrt{\sum_{i=1}^{k}  i(\mu_{i},\cdot)^{2}}
			\ \ \ \mathrm{and}\ \ \
			i(\eta, \cdot) =
			\sqrt{\sum_{j=1}^{l}  i(\nu_{j},\cdot)^{2}}.
		\end{equation*}
		Then the foliations ${F}_{+},{F}_{-}$ have ergodic decompositions 
		\begin{equation*}
			{F}_- = \sum_{i=1}^{k} x_i \mu_i\ \ \
			\mathrm{and}\ \ \
		{F}_+ = \sum_{j=1}^{l} y_j \nu_j,
		\end{equation*}
		where the unknown coefficients $x_i>0, y_j>0$ and they must satisfy the equations
		\begin{equation*}
			i(\mu_i, \cdot)^2
			= \frac{x_i i(\mu_i, \cdot)^2}{i(\mu_i,{F}_+)}\ \ \
			\mathrm{and}\ \ \
			i(v_j, \cdot)^2
			= \frac{y_j i(\nu_j, \cdot)^2}{i(\nu_j, {F}_-)},
		\end{equation*}
		i.e.,
		\begin{equation*}
			x_i
			= {i(\mu_i, {F}_+)}\ \ \
			\mathrm{and}\ \ \
			y_j
			= {i(\nu_j, {F}_-)}.
		\end{equation*}
		
		Thus we obtain a system of linear equations:
		\begin{equation}\label{equ:system}
			\left\{
			\begin{array}{lr}
				x_i =  i(\mu_i, {F}_+) =  \sum_{j=1}^{l} y_j i(\mu_i, \nu_j) =  \sum_{j=1}^{l} y_j n_{ij} \\
				\\
				y_j = i(\nu_j, {F}_-) = \sum_{i=1}^{k} x_i i(\nu_j, \mu_i) = \sum_{i=1}^{k} x_i n_{ij}
			\end{array}
			\right.
		\end{equation}
		where $1 \leq i \leq k$, $1 \leq j \leq l$ and  $n_{ij} = i(\mu_i, \nu_j)$.
		Denote the matrix of intersection by
		
		$$N=\left( n_{ij} \right)_{1\leq i\leq k, 1\leq j\leq l}.$$
		
		Now we
		can reduce the construction of Teichm\"uller geodesic to 
		the problem of  finding  a vector $\mathbf{x}=(x_1, \cdots, x_k)$ such that
		$$\mathbf{x}= N N^* \ \mathbf{x}, \ \mathrm{where} \  \mathbf{x}>0.$$
		(All the entries $x_i$ are positive.)
		
			In general, the solution does not exist. However, note that the representation of
		a Busemann point is defined up to multiplication by a positive constant.
		Thus it remains to show that the matrix $T=N N^*$ has a positive eigenvector.	This is trivial when $k=1$. Thus we assume that $k\geq 2$.

			To see this, we observe that the matrix $T$ is non-negative (all the entries
		of $T$ are non-negative), symmetric and positive semi-definite. Thus all the eigenvalues of $T$ are non-negative. Denote the maximal eigenvalue of $T$
		by $\lambda$. Then $\lambda>0$ (otherwise, $T\equiv 0$).

		Now we use the assumption that $F_+$ and $ F_-$ fill up the surface to show that the matrix $T$
		is primitive, that is, some power of $T$ is positive.
		
	In the case that  ${F}_-$ and ${F}_+$ are rational,
		i.e., all the components are weighted simple closed curves,
	$T$ is primitive if and only if the union of the support of ${F}_-$ and ${F}_+$
		is a connected graph on the surface  (see Thurston \cite[Section 6]{thurston1988geometry}
		for more precise description).
		Here, the  connectedness is confirmed by filling condition.
		
		In general, we can use a theorem of Lezhen-Masur \cite[Theorem C]{lenzhen2010criteria} to approximate the foliations $\sum_{i=1}^k \mu_i$
		and $\sum_{j=1}^l \nu_j$ by a pair of rational measured foliations
		$\alpha=\sum_{i=1}^k c_i \alpha_i$ and $\beta=\sum_{j=1}^l d_j \beta_j$,
		which satisfy the following properties:
		\begin{itemize}
			\item  $\alpha$ and $\beta$ fill up the surface.
			\item  $i(\mu_i,\mu_j)=0$ if and only if $i(\alpha_i,\beta_j)=0$.
		\end{itemize}
		By a similar argument, we can show that the matrix $T$ is primitive.
		
		We can now proceed by using the Perron-Frobenius theorem to conclude that there exists some positive vector $\mathbf{x}> 0$
		such that $T \mathbf{x}=\lambda \mathbf{x}$. Moreover, $\lambda$ is the only eigenvalue of $T$
		with a non-negative eigenvector and $\lambda$ is a simple eigenvalue.
		
		By our discussion, 
		the existence of $\mathbf{x}$ implies the existence of $\mathbf{G}_{q}$,
		and the uniqueness of the geodesic follows from the uniqueness of $\mathbf{x}$
		(up to a positive constant). This completes the proof.
	\end{proof}

	We obtain an analogue of Theorem  \ref{thm:geodesic}.

	\begin{corollary}
		Let $\xi, \eta \in \partial_{\mathrm{GM}}\T$ be two Busemann points. Then there is a Teichm\"uller geodesic  $\mathbf{G}(\cdot): \mathbb{R} \rightarrow \mathcal{T}(S)$ such that
		\begin{equation*}
			\lim_{t \rightarrow +\infty}\mathbf{G}(t) = \xi,\ \ \lim_{t \rightarrow - \infty}\mathbf{G}(t) = \eta
		\end{equation*}
		if and only if
		$\xi,\eta$ 
		fill up the surface.
		Moreover, the geodesic is unique. 
	\end{corollary}

	\section{Optimal geodesics for horofunctions}
	
	In this section we prove Theorem \ref{thm:horofunction}.
	We  shall identify
	$\partial_{\mathrm{GM}} \mathcal{T}(S)$ with $\partial _h \mathcal{T}(S)$ through the homeomorphism $\Psi$ (see Theorem \ref{thm:LS} for definition). 
	
	Recall that
	a Teichm\"uller geodesic $\mathbf{G}(t)$ is optimal for $\xi \in \GMT$ if,
	for all $t \in \mathbb{R}$, 
	$$
	\Psi_\xi(\mathbf{G}(t))-\Psi_\xi(\mathbf{G}(0))=-t.
	$$
	In particular, the geodesic $\mathbf{G}$
	is optimal for its limited Busemann point
	$\lim\limits_{t\to +\infty} \mathbf{G}(t)$.
	
	The existence of optimal geodesic for a single horofunction is proved by
	Azemar \cite[Proposition 3.3]{Azemar2021qualitative}.
	
	\begin{lemma}[Azemar]\label{lem-exist}
		Fix $X_0\in \TS $. For each $\xi \in \GMT$, there is  a unique optimal geodesic $\mathbf{G}(t)$ for $\xi$
		with $\mathbf{G}(0)=X_0$. 
		Moreover, let $\{X_{n}\} \subset \mathcal{T}(S)$ be any sequence that converges to $\xi$.
		Let $\mathbf{G}_n$ be the Teichm\"uller geodesic associated to
		 $q_n\in Q^1(X_0)$ such that $\mathbf{G}_n$ passing through
		$X_n$ in the positive direction. Then $q_n$ converges to some $q\in Q^1(X_0)$
		and  the Teichm\"uller geodesic $\mathbf{G}=\mathbf{G}_q$          is optimal  for $\xi$.
	\end{lemma}
	
	Now we have a pair of boundary points $\xi$
	and $\eta$ that fill up the surface. 
	The main idea to prove the existence in Theorem \ref{thm:horofunction} is as following. 
	Given $X_0\in \TS$, we have geodesics 
	$\mathbf{G}_1$ and $\mathbf{G}_2$, passing through $X_0$, which is optimal for $\xi$
	and $\eta$, respectively. 
	If  the geodesics $\mathbf{G}_1$ and $\mathbf{G}_2$ overlap, then it
	is the optimal geodesic that we want.

	\begin{figure}[htb]
		\begin{tikzpicture}
			\draw[thick] (0,0.1) -- (5,1);
			\draw [dashed] (0,0.1) -- (-5,-0.9); 
			\draw[thick] (0, 0.1) -- (-5,-2);
			\draw [dashed] (0,0.1) -- (5,1.9);
			\node[blue] at (0,0.1) {$\bullet$ };
			\node at (0,-0.3) {$X_0$};
			\node at (4,0.4) {$\mathbf{G}_1$};
			\node at (-4,-2) {$\mathbf{G}_2$};
		\end{tikzpicture}\caption{Assume that $X_0=\mathbf{G}_1(0)=\mathbf{G}_2(0)$. The geodesics $\mathbf{G}_1$ and $\mathbf{G}_2$ may not overlap. } 
		\label{fig_01}
	\end{figure}
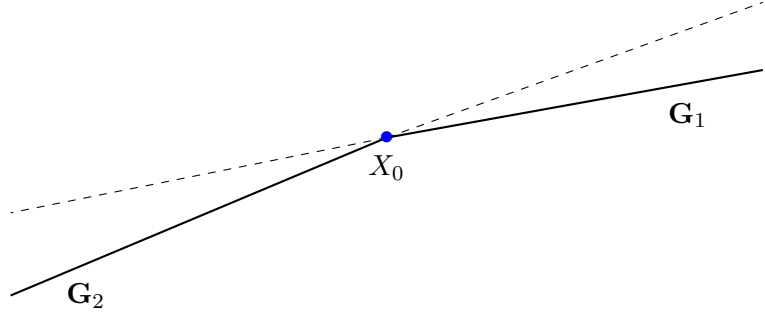

Unfortunately,  if the point $X_0$ is not correctly chosen, 
	then  $\mathbf{G}_1$ and $\mathbf{G}_2$ may not overlap.
	The key to the proof is that $X_0$ must be an intersection point of two horospheres that are tangent to each other. 
	
	Horospheres are level sets of horofunctions. Some of their properties are investigated in
	\cite{Su2018Horospheres}.

	\subsection{Upper and lower bounds for horofunctions.}
	
	By Lemma \ref{lem-exist}, for each horofunction $\xi \in \GMT$,
	we can  associate  with an optimal geodesic $\mathbf{G}=\mathbf{G}_q$ with $\mathbf{G}(0)=X_0$. We shall denote the
	limited Busemann function (in the positive direction) of $\mathbf{G}$ by
	$B_{\xi}$.

	The Busemann function $B_\xi$ provides an upper bound
	for the horofunction corresponding to $\xi$. 
	The next lemma was proved by  Azemar \cite[Proposition 3.9]{Azemar2021qualitative}. See also \cite[Lemma 3.10]{LiuShi}.
	
	\begin{lemma}\label{Lem:upper_bounded}
		For every $X \in \mathcal{T}(S)$,  we have
	$\Psi_\xi(X) \leq B_{\xi}(X).
$
	\end{lemma}
	
	Let $F=F_v(q)$ be  the vertical measured foliation of $q$.
	We consider the projective class of $F$ as an element of the Gardiner-Masur boundary. The corresponding horofunction will be denoted by
	$\Psi_{F}$. By definition, 
	
	$$\Psi_{F}(X)=\log \sup _{\mu \in \mathcal{MF}} \frac{i({F},\mu)}{\sqrt{\operatorname{Ext}_{X}(\mu)}}
	-
	\log \sup _{\mu \in \mathcal{MF}} \frac{i({F},\mu)}{\sqrt{\operatorname{Ext}_{X_0}(\mu)}}.
	$$
	
	\begin{lemma}\label{lem_up_bound}
		With the above notation, for all $X \in \mathcal{T}(S)$, we have
		$$\Psi_{{F}}(X)\leq \Psi_\xi(X). $$
	\end{lemma}
	\begin{proof}
		Note that
		$$
		\Psi_\xi(X)=\log \sup _{\mu \in \mathcal{MF}} \frac{i(\xi, \mu)}{\sqrt{\operatorname{Ext}_{X}(\mu)}}
		-
		\log \sup _{\mu \in \mathcal{MF}} \frac{ i(\xi,\mu)}{\sqrt{\operatorname{Ext}_{X_0}(\mu)}}.$$
		It is not hard to check that  (see Remark \ref{remark:one})
		$$
		\log \sup _{\mu \in \mathcal{MF}} \frac{i(\xi, \mu)}{\sqrt{\operatorname{Ext}_{X_0}(\mu)}}=
		\log \sup _{\mu \in \mathcal{MF}} \frac{i(F,\mu)}{\sqrt{\operatorname{Ext}_{X_0}(\mu)}}=0.
		$$
		Moreover, by Proposition \ref{prop_low-bound}, $i({F},\mu)\leq
		i(\xi, \mu)$ for all $\mu\in \mathcal{MF}$, which implies that
		$$\Psi_{{F}}(X)\leq \Psi_{\xi}(X).$$
	\end{proof}
	
	The above two lemmas give sharp upper and lower bound of $\Psi_\xi$, in terms of the data associated to the optimal geodesic. We remark that $\Psi_\xi$, $\Psi_{{F}}$
	and $B_{\xi}$ are comparable to each other, that is,
	$$\frac{B_\xi}{C} \leq \Psi_{F} \leq \Psi_\xi \leq B_\xi,$$
	where $C$ is a constant depends on $X_0$.
	In the following, we shall denote by $\Psi_F \asymp \Psi_\xi \asymp B_\xi$.

	\subsection{Level sets and sub-level sets}
	We need some properties on the level sets of horofunctions.
	
	\begin{definition}[Level set and sub-level set]
		Let $f$ be a function defined on $\T$. For each $r \in \mathbb{R}$,
		the level set and sub-level set of $f$ are defined by
		$$
		\mathrm{L}(f,r) = \{X \in \mathcal{T}(S)\ | \ f(X) = r\}
		$$
		and
		$$
		\mathrm{SL}(f,r) = \{X \in \mathcal{T}(S)\ |\ f(X) < r\}.
		$$
	\end{definition} 
	Assume that $Y \in \T$ and it is contained in a level set
	$\mathrm{L}(f, r)$, that is, $f(Y)=r$. Then we also let
	$$\mathrm{L}(f,Y)=\{X\in \T \ | \  f(X)=f(Y)=r\}$$
	denotes the level set passing through $Y$.  
	
	When $f$ is a horofunction function, level sets of $f$ are called  \textit{horospheres} and sub-level sets of $f$ are called \textit{horoballs}. 
	
	\bigskip
	
	In particular, let $F \in \mathcal{MF}$ be a measured foliation.
	The corresponding horofunction is given by
	$$
	\Psi_{F}(X) = \frac{1}{2}\log \sup _{\mu \in \mathcal{MF}} \frac{i({F}, \mu)^2}{\operatorname{Ext}_{X}(\mu)}
	-
	\frac{1}{2} \log \sup _{\mu\in \mathcal{MF}} \frac{i({F}, \mu)^2}{\operatorname{Ext}_{X_0}(\mu)}.
	$$
	We may assume that $\mathrm{Ext}_{X_0}({F}) = 1$. Then,
	by an inequality of Gardiner-Masur \cite{gardiner1991extremal},
	$$
	\Psi_{{F}}(X)= \frac{1}{2}\log \operatorname{Ext}_{X}({F}).
	$$
	This implies that
	\begin{equation*}\label{eq_03}
		\mathrm{L}(\Psi_{{F}},r) = \{X \in \mathcal{T}(S)\ |\ \mathrm{Ext}_{X}({F}) = e^{2r}\}.
	\end{equation*}

	\medskip

	\begin{definition}
		Let $\xi,\eta \in \partial_{h}\mathcal{T}(S)$ be two horofunctions. We say that the level sets
		$\mathrm{L}(\Psi_\xi,r)$ and $\mathrm{L}(\Psi_\eta,s)$ are tangent to each other at $X\in \T$ if
		
		\begin{itemize}
			\item $X\in \mathrm{L}(\Psi_\xi,r)\cap \mathrm{L}(\Psi_\eta, s)$.
			\item $\mathrm{SL}(\Psi_\xi,r)$ and $\mathrm{SL}(\Psi_\eta,s)$ are disjoint.
		\end{itemize}
	\end{definition}
We remark that the above definition doesn't assume that $\mathrm{L}(\Psi_\xi,r)\cap \mathrm{L}(\Psi_\eta, s)$ consists of a single point.

	\begin{lemma}\label{lem:tangent}
		Let $\xi$ and $\eta$ be a pair of horofunctions that fill up the surface. 
		For any given $s_0 \in \mathbb{R}$, there exists some $t_0\in \mathbb{R}$ (depending on $s_0$) such that $\mathrm{L}(\xi,s_0)$
		and $\mathrm{L}(\eta,t_0)$ are tangent to each other.
	\end{lemma}
	\begin{proof}
		Let $\mu$ and $\nu$ be the vertical measured foliations
		of the quadratic differentials that define the optimal geodesics
		for $\xi$ and $\eta$, respectively. By our assumption,
		$\mu$ and $\nu$ also fill up the surface.
		
		By \cite[Lemma 3.10]{Su2018Horospheres}, for given $s_0$,
		there is a constant $t_0$ such that
		\begin{enumerate}
			\item for all $t<t_0$, $\mathrm{SL}(\Psi_{\nu},t)$ is disjoint from $\mathrm{SL}(\Psi_\mu,s_0)$.
			\item for all $t\geq t_0$, the intersection of $\overline{\mathrm{SL}(\Psi_{\nu},t)}$ and $\overline{\mathrm{SL}(\Psi_\mu,s_0)}$ is non-empty and compact.
		\end{enumerate} 
		Since $\Psi_\xi\asymp \Psi_{\mu}$ and $\Psi_\eta \asymp \Psi_{\nu}$,
		the same statement holds if we replace $\Psi_{\mu}$ and $\Psi_{\nu}$
		by  $\Psi_\xi$ and $\Psi_\eta$.

		As a result, there is some maximal number $t_0$ such that for all $t<t_0$, $\mathrm{SL}(\Psi_\eta,t)$ is disjoint from $\mathrm{SL}(\Psi_\xi,s_0)$. Then $\mathrm{L}(\Psi_\xi,s_0)$
		and $\mathrm{L}(\Psi_\eta,t_0)$ are tangent to each other. 
	\end{proof}

	\subsection{ Busemann functions are $C^1$.}

	Note that for any measured foliation $\mathcal{F}$, the horofunction $\Psi_{\mathcal{F}}$ is $C^1$.
	This follows from the fact that extremal length functions are $C^1$.
	
	\begin{lemma}
		Let $B$ be a Busemann function. Then $B$ is $C^1$ in the Teichm\"uller space.
	\end{lemma}
	\begin{proof}
		Consider any $X\in \T$.  Let $\mu$ be a Beltrami differential on
		$X$, and let $X_{s \mu}$ be the quasiconformal deformation of $X$ with Beltrami differential
		$s \mu$,  where $-\epsilon<s<\epsilon$.
		
		We may assume that $B(X)=0$. There is a unique quadratic differential $q=q(X)\in Q^1(X)$ such that
		the geodesic ray induced by $q$ converges to the Busemann point (function) $B$.
		Take some $r>0$. Then the above geodesic intersects the geodesic sphere of radius $r$ (centered at $X$)  at two points, denoted by $X_1$ and $X_2$. We assume that $B(X_1)=-r, B(X_2)=r$.

		Define two auxiliary functions
		$$A_1(s)=d_{\mathcal{T}}(X_{s\mu}, X_1)-r,$$
		$$A_2(s)=-d_{\mathcal{T}}(X_{s\mu}, X_2)+r.$$
		Note that $A_1(0)=A_2(0)=0$.
		
		Since the Busemann function is $1$-Lipschitz
		and $B(X_1)=-r, B(X_2)=r$, we have
		\begin{eqnarray*}
			B(X_{s\mu})+ r &\leq& d_{\mathcal{T}}(X_{s\mu},X_1) \\
			-d_{\mathcal{T}}(X_{s\mu},X_2)&\leq& B(X_{s\mu})- r.
		\end{eqnarray*}
		Thus
		$$A_2(s)\leq B(X_{s\mu})\leq A_1(s).$$
		
		Note that $A_1,A_2$ are differentiable at $s=0$. In fact,  there is a formula due to Earle \cite{Earle1977the}:
		\begin{eqnarray*}
			\lim_{s\to 0} \frac{A_1(s)}{s}&=&\lim_{s\to 0}\frac{d(X_{s\mu},X_1)}{s} \\
			&=& \operatorname{Re} \int_X \mu (-q)=-\operatorname{Re} \int_X \mu q.
		\end{eqnarray*}
		Similarly,

		\begin{eqnarray*}
			\lim_{s\to 0} \frac{A_2(s)}{s}&=& - \lim_{s\to 0}\frac{d(X_{s\mu},X_2)}{s} \\
			&=&-\operatorname{Re} \int_X \mu q.
		\end{eqnarray*}
		As a result, we obtain
		
		$$-\operatorname{Re} \int_X \mu q \leq \liminf_{s\to 0} \frac{B(X_{s\mu})}{s}\leq \limsup_{s\to 0} \frac{B(X_{s\mu})}{s}\leq -\operatorname{Re} \int_X \mu q.$$
		This implies that
		$$\lim_{s\to 0} \frac{B(X_{s\mu})}{s}= - \operatorname{Re} \int_X \mu q.$$

		\bigskip
		
		We have shown that the Busemann function is differentiable at every $X\in\T$,
		and the differential of $B$ is given by the contangent vector $q=q(X)$.
		
		We conclude the proof by showing that $q(X)$ depends continuously on $X$.
		In fact, for any sequence $X_n \in \T$ that converges to $X$,
		$q_n=q(X_n)\in Q^1(X_n)$ is the quadratic differential which induces the geodesic ray starting from
		$X_n$ and converging to the Busemann point $B$.
		It was proved by Walsh \cite[Theorem 4]{Walsh2019The} that each $q_n$ is modular equivalent to $q=q(X)$.
		
		Assume that $\sum_j a^n_j {F}_j$ is the ergodic decomposition of the vertical foliations
		of $q_n$. Denote by $H(q_n)$ the horizontal foliations of $q_n$. Then, by definition,
		$$\frac{a_j^n}{i({F}_j, H(q_n))}=c_j,$$
		where $c_j>0$ is a constant independent of $n$.
		Up to a subsequence, we may assume that $q_n$ converges to some $\phi\in Q^1(X)$.
		Then $$i({F}_j, H(q_n))\to i({F}_j, H(\phi)).$$
		On the other hand, the vertical foliation of $\phi$ is given by $\sum_j a_j {F}_j$,
		where $a_j=\lim_{n\to\infty} a^n_j \geq 0$.
		Note that $i({F}_j, H(\phi))$ is the length of ${F}_j$ in the flat metric $|\phi|$, which must be positive. This implies that $a_j>0$.
		
		As a result, we have shown that $\phi$ is modular equivalent to each $q_n$, and then it is modular equivalent to
		$q$. Since on $X$ the modular equivalence class in $Q^1(X)$ is unique \cite[Theorem 5]{Walsh2019The}, we have $q=\phi$.
		
	\end{proof}

	\begin{remark}
		We don't know whether any horofunction is $C^1$ in the Teichm\"uller space or not. 
	\end{remark}
	
	\subsection{Existence of optimal geodesic.}

	\begin{theorem}\label{thm_01}
		Let $\xi, \eta \in \GMT$ that fill up the surface. Then there exists a Teichm\"uller geodesic
		$\mathbf{G}(t)$ that is optimal for $\xi$ and $\eta$. That is, along the geodesic 	$\mathbf{G}(t)$, 
		
		\begin{itemize}
			\item $\Psi_\xi(\mathbf{G}(t)) - \Psi_\xi(\mathbf{G}(0)) = -t$, and 
			\item $\Psi_\eta(\mathbf{G}(t)) - \Psi_\eta(\mathbf{G}(0)) = t$.
		\end{itemize}
	\end{theorem}

	\begin{proof}
		By  Lemma \ref{lem:tangent}, there are
		some $s_0, t_0 \in \mathbb{R}$ such that the level sets $\mathrm{L}(\xi, s_0)$ and $\mathrm{L}(\eta, t_0)$ are tangential to each other at some $X_0\in \TS$.
		
		In the absence of any confusion, we take $X_0$ to be the basis-point of the horofunction compactification.
		In this case,  $s_0=t_0=0$.

		Let $\mathbf{G}_{1}(t)$ be the optimal geodesic ray for $\xi$ and let $\mathbf{G}_{2}(t)$ be the optimal geodesic ray for $\eta$, where $\mathbf{G}_1(0)=\mathbf{G}_2(0)=X_0$.
		We claim that
			$\mathbf{G}_{1}(t) = \mathbf{G}_{2}(-t).$

		To see this, assume that the geodesic $\mathbf{G}_{i}$ is induced by $q_i\in Q^1(X_0)$.
		Denote the vertical foliation
		of $q_i$ by ${F}_i$. And we denote the Busemann function associated to $q_i$ by $B_i$.
		Then we have
		$$\Psi_{{F}_i}(X_0)=\Psi_{B_i}(X_0)=0.$$
		According to Lemma \ref{Lem:upper_bounded} and Lemma \ref{lem_up_bound}, we have
		$$
		\mathrm{SL}(B_1, 0)\subset  \mathrm{SL}(\Psi_\xi,0) \subset \mathrm{SL}(\Psi_{{F}_1}, 0),$$
		and
		$$
		\mathrm{SL}(B_2, 0)  \subset \mathrm{SL}(\Psi_\eta,0) \subset \mathrm{SL}(\Psi_{{F}_2}, 0).
		$$
		
		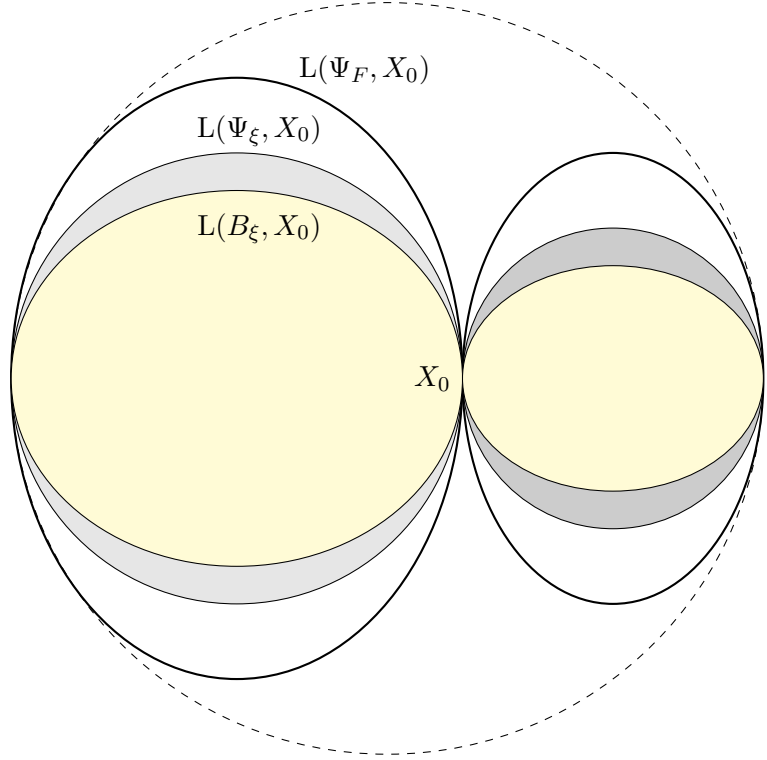
\begin{figure}[htb]
		\begin{tikzpicture}
			\draw[dashed] (0,0) circle (5cm);
			\filldraw[fill=gray!40] (3,0) circle (2cm);
			\draw[thick] (3,0) ellipse (2cm and 3cm);
			\filldraw[fill=gray!20] (-2,0) circle (3cm);
			\draw[thick] (-2,0) ellipse (3cm and 4cm);
				\filldraw[fill=yellow!20] (-2,0) ellipse (3cm and 2.5cm);
			\node at (-0.3,4.1) {$\mathrm{L}(\Psi_{{F}},X_0)$};
			\node at (-1.7,3.3) {$\mathrm{L}(\Psi_\xi,X_0)$};
				\node at (-1.7,2) {$\mathrm{L}(B_\xi,X_0)$};
					\filldraw[fill=yellow!20] (3,0) ellipse (2cm and 1.5cm);
			\node at (0.6,0) {$X_0$};
		\end{tikzpicture}\caption{Proof of Theorem \ref{thm_01}. 
		It follows from the inequality 
		$\Psi_{F}\leq \Psi_\xi\leq B_\xi$ that 
		$
		\mathrm{SL}(B_\xi, 0)\subset  \mathrm{SL}(\Psi_\xi,0) \subset \mathrm{SL}(\Psi_{{F}}, 0),$} 
		\label{fig_02}
	\end{figure}
		
		We need

		\begin{lemma}\label{lem:ab}
			For any $C^1$-horofunction $f$, denote by $V(f)$ the space of tangent vectors $\mu$ at $X_0$ such that
			$df(\mu)=0.$
			Then $V(B_1)=V(B_2)$.
		\end{lemma}
		\begin{proof}
			For each $\mu\in V(B_1)$, there is a path $\gamma(t)\subset \T, -\epsilon<t<\epsilon$ that is tangent to $\mu$
			at $\gamma(0)=X_0$ and $B_1(\gamma(t))\equiv 0$. Since $\mathrm{L}(B_1,0)$ and $\mathrm{L}(B_2,0)$ are tangent to each other
			at $X_0$, we have
			$$B_2(\gamma(t))\geq 0, B_2(X_0)= 0.$$
			It follows that $X_0$ is a point of minimum  of $B_2(\gamma(t))$. Thus
			$\mu\in V(B_2)$.
			
			We have shown that $V(B_1)\subset V(B_2)$.
			Similarly, we have $V(B_2)\subset V(B_1)$.
			
		\end{proof}
		
		By the same argument, we can show that
		$V(B_i)=V(\Psi_{{F}_i}).$
		This implies that
		$$
		V(\Psi_{{F}_1})=V(\Psi_{{F}_2}).
		$$
		
		The next lemma is proved by \cite[Lemma 3.5]{Su2018Horospheres}:
		
		\begin{lemma}\label{lem:coincide}
			Let $q\in Q^1(X_0)$ be the quadratic differential
			with vertical foliation ${F}_v$ and horizontal foliation
			${F}_h$.  Then $F$ is a measured foliation satisfying $V(\Psi_{{F}_v})=V(\Psi_{{F}})$ if and only if ${F}$ is projectively equivalent to ${F}_v$ or 
			${F}_h$.
		\end{lemma}
		
	By Lemma \ref{lem:coincide}, we have  $q_1=-q_2$ and then $\mathbf{G}_{1}(t) = \mathbf{G}_{2}(-t).$
	\end{proof}

	\subsection{Uniqueness of optimal geodesic.}
	
	Let $\xi,\eta$ be a pair of horofunctions that fill up the surface. We have shown that there exists at least one geodesic which is optimal for $\xi$
	and its inverse is optimal for $\eta$.  For simplicity, we call such a geodesic \emph{optimal} for 
	$\xi$ and $\eta$. 
	
	\begin{figure}[htb]
		\begin{tikzpicture}
			\draw[dashed] (0,0) circle (5cm);
			\filldraw[fill=gray!40] (3,0) circle (2cm);
			\draw[thick] (3,0) ellipse (2cm and 2.9cm);
			\filldraw[fill=gray!20] (-2,0) circle (3cm);
			\draw[thick] (-2,0) ellipse (3cm and 3.9cm);
			\node at (-2,0) {$\mathrm{SL}(\Psi_\xi,X)$};
			\node at (3,0) {$\mathrm{SL}(\Psi_\eta,X)$};
			\node at (-0.3,4.1) {$\mathrm{L}(\Psi_{\mathcal{F}_v},X)$};
			\node at (-1.5,2.1) {$\mathrm{L}(\Psi_\xi,X)$};
			\node at (2,-3.4) {$\mathrm{L}(\Psi_{\mathcal{F}_h},X)$};
			\node at (3,-1.2) {$\mathrm{L}(\Psi_\eta,X)$};
			\node at (0.8,0) {$X$};
		\end{tikzpicture}\caption{Proof of Lemma \ref{lem:seperate}. 
			The  horoballs $\mathrm{SL}(\Psi_\xi,X)$ and $\mathrm{SL}(\Psi_\eta,\eta)$ are seperated by the horospheres $\mathrm{L}(\Psi_{\mathcal{F}_v},X)$ and $\mathrm{L}(\Psi_{\mathcal{F}_h},X)$.  } 
		\label{fig_01}
	\end{figure}
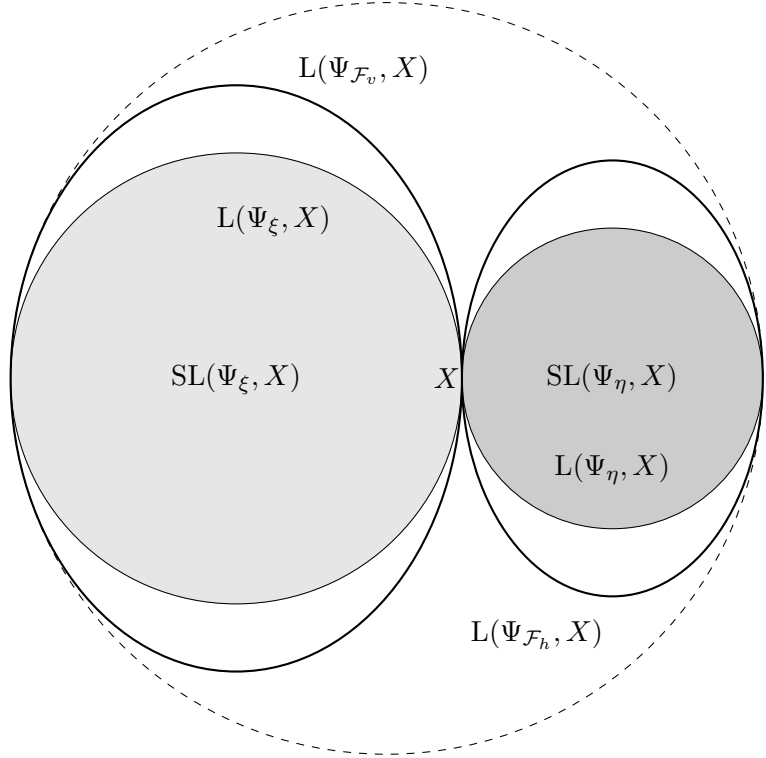
	
	\begin{lemma}\label{lem:seperate}
		Let $\mathbf{G}$ be an optimal geodesic for $\xi$ and $\eta$. Then for any 
		$X\in \mathbf{G}$, the level set $\mathrm{L}(\Psi_\xi,X)$ and $\mathrm{L}(\Psi_\eta,X)$ are tangent to each other at $X$. Moreover,
		$$\mathrm{L}(\Psi_\xi,X)\cap\mathrm{L}(\Psi_\eta,X)=\{X\}.$$
	\end{lemma}
	\begin{proof}
		Assume that the geodesic $\mathbf{G}$ is directed by  $q\in Q^1(X)$.
		Denote the horizontal and vertical foliations of $q$ by $\mathcal{F}_h$
		and $\mathcal{F}_v$. 
		
		By the results in \cite{Su2018Horospheres}, the level sets $\mathrm{L}(\Psi_{\mathcal{F}_v},X)$ and $\mathrm{L}(\Psi_{\mathcal{F}_h},X)$ are tangent to each other at $X$.
		Moreover, $X$ is the unique intersection point of $\mathrm{L}(\Psi_{\mathcal{F}_v},X)$ and $\mathrm{L}(\Psi_{\mathcal{F}_h},X)$. 
		
		Since $\Psi_{\mathcal{F}_v}\leq \Psi_\xi$ and $\Psi_{\mathcal{F}_h}\leq \Psi_\eta$, 
		the  sub-level sets $\mathrm{SL}(\Psi_\xi,X)$ and $\mathrm{SL}(\Psi_\eta, X)$ are seperated by $\mathrm{L}(\Psi_{\mathcal{F}_v},X)$ and $\mathrm{L}(\Psi_{\mathcal{F}_h},X)$, as shown in the Figure \ref{fig_01}.  As a result, $\mathrm{L}(\Psi_\xi,X)$ and $\mathrm{L}(\Psi_\eta,X)$ are tangent to each other at
		$X$, which is unique. 
		
	\end{proof}

	Now we can verify 
	\begin{theorem}\label{thm:unique}
		For any pair  of horofunctions $\xi,\eta$   that fill up the surface, the optimal geodesic for 
		$\xi$ and $\eta$ is unique. 
	\end{theorem}
	\begin{proof}
		
		Assume that both $\mathbf{G}_1$ and $\mathbf{G}_2$ are optimal geodesics for $\xi$ and $\eta$. 
		Then, using the proof of Theorem \ref{thm_01}  and Lemma \ref{lem:seperate}, we know that
		for any fixed $s \in \mathbb{R}$, there exists a unique $t \in \mathbb{R}$ such that level sets $\mathrm{L}(\Psi_\xi,s)$ and $\mathrm{L}(\Psi_\eta,t)$ are tangent to each other at some $X_1\in \mathbf{G}_1$
		and at some $X_2\in \mathbf{G}_2$. 
		
		By Lemma \ref{lem:seperate}, we have $X_1=X_2$. Again, the assumption that
		$\mathrm{L}(\Psi_\xi,s)$ and $\mathrm{L}(\Psi_\eta,t)$ are tangent to each other 
		implies $\mathbf{G}_1=\mathbf{G}_2$.
		
	\end{proof}

	Theorem \ref{thm:horofunction} follows from Theorem \ref{thm_01}
	and Theorem \ref{thm:unique}.

	\section{Convergence of geodesics}
	
	In this section, we prove Theorem \ref{thm:convergence}. 
	Fix a basis-point $X_0\in \TS$.
	For simplicity, we denote $d=d_{\mathcal{T}}$.
	
	Let  $\xi, \eta$ be a pair of Gardiner-Masur boundary points that  fill up the surface. In particular, $i(\xi,\eta)\neq 0$. 
	Consider any two sequences $X_n,Y_n\in\TS$  such that $X_n\to\xi$ and $Y_n\to\eta$.
	Denote the Teichm\"uller geodesics passing through $X_n$ and $Y_n$ by  $\mathbf{G}_n=\mathbf{G}_n(t)$.
	
	We first prove
	
	\begin{lemma}\label{lem:faster}
		The distance $d(X_n, Y_n)$ growths faster than both
		$d(X_0, X_n)$  and $d(X_0, Y_n)$.
	\end{lemma}

	\begin{proof}
		We recall Miyachi's intersection form (see Theorem \ref{thm:Miyachi})
		$$\log i(X,Y)= d(X,Y)- d(X_0, X)-d(X_0, Y). $$

		Since $i(\xi,\eta)>0$ and $X_n\to\xi, Y_n \to \eta$, we have

		\begin{eqnarray*}
			&& \lim_{n\to\infty}	\left(d(X_n, Y_n) -d (X_0, X_n) - d (X_0, Y_n) \right)\\
			&=& \log i (\xi,\eta)  \\
			&>& -\infty. 
		\end{eqnarray*}
		
		As a result, there is a constant $c$ such that, for $n$ sufficiently large,

		\begin{eqnarray*}
			d(X_n, Y_n) -d (X_0, X_n) \geq  d (X_0, Y_n) + c, \\
		\end{eqnarray*}
		and, 
		\begin{eqnarray*}
			d(X_n, Y_n) -d (X_0, Y_n) \geq  d (X_0, X_n) + c.\\
		\end{eqnarray*}
		
		Thus, we have
		\begin{eqnarray*}  && \lim_{n\to\infty} \left(d(X_n, Y_n) -d (X_0, X_n) \right) \\
			&=& \lim_{n\to\infty} \left(d(X_n, Y_n) -d (X_0, Y_n) \right)  \\
			&=& \infty
		\end{eqnarray*}
	\end{proof}

	\begin{remark}
		We can conjecture that if $\xi$ and $\eta$ fill up a subsurface of $S$,
		then the distance $d(X_n, Y_n)$ growths faster than 
		$d(X_0, X_n)$ and $d(X_0, Y_n)$. See \cite[Page 545]{DF} for the conjecture when 
		$\xi$ and $\eta$ are simple closed curves with distance $2$ in the complex of curves. 
	\end{remark}

	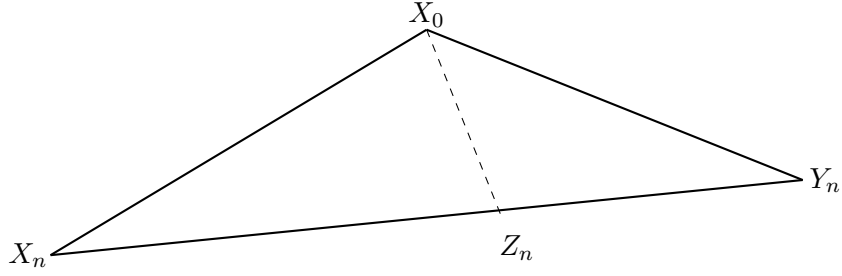
\begin{figure}[htb]
		\begin{tikzpicture}	
			\draw[thick] (-5,0) -- (5,1);
			\draw[thick] (-5,0) -- (0,3);
			\draw[thick] (0,3) -- (5,1);
			\draw [dashed] (0,3) -- (1,0.5);
			\node at (-5.3,0) {$X_n$};
			\node at (0,3.2) {$X_0$};
			\node at (5.3, 1) {$Y_n$};
			\node at (1.2, 0.1) {$Z_n$};
		\end{tikzpicture}\caption{The sequence of geodesics passing through $X_n$
			and $Y_n$.} 
		\label{fig_02}
	\end{figure}

	With Lemma \ref{lem:faster}, for $n$ sufficiently large, we can take some $Z_n \in \mathbf{G}_n $
	such that $d(X_n, Z_n)= d(X_0, X_n)$. 
	We claim 
	\begin{lemma}\label{lem:compact}
		There is some compact set $\mathcal{K}\subset \TS$ that contains all 
		$Z_n$. 
	\end{lemma}

	\begin{proof}
		It suffices to check that $d(X_0, Z_n)$ is uniformly bounded. On the contrary, suppose that $Z_n$ diverges in the Teichm\"uller space. Then, up to a subsequence, we may assume that $Z_n\to \lambda$, where $\lambda$ is some boundary point. 
		
		By our construction, 
		
		\begin{eqnarray*}
			\log i(X_n,Z_n) &=& d(X_n,Z_n)- d(X_0, X_n)-d(X_0, Z_n) \\
			&=& -d(X_0, Z_n).
		\end{eqnarray*}
		
		By the assumption, 	$d(X_0, Z_n) \to \infty$. This implies that 
		
		$$i(\xi, \lambda)= \lim_{n\to\infty} i(X_n,Z_n) = 0.$$

		Similarly, 
		
		\begin{eqnarray*}
			&&	\log i(Y_n,Z_n) -	\log i(X_n,Y_n)  \\
			&=& d(Y_n,Z_n)- d(X_0, Y_n)-d(X_0, Z_n) \\
			&& - \left( d(X_n,Y_n)- d(X_0, X_n) - d(X_0, Y_n) \right) \\
			&=&	 - d(X_n, Z_n)- d(X_0, Z_n) + d(X_0, X_n) \\
			&= & - d(X_0, Z_n).
		\end{eqnarray*}
		That is,
		
		\begin{eqnarray*}
			\log i(Y_n,Z_n) 
			&=& - d(X_0, Z_n) + 	\log i(X_n,Y_n).
		\end{eqnarray*}
		Since $	\log i(X_n,Y_n)$ is uniformly bounded and $d(X_0, Z_n) \to \infty$,   we have 
		$$i(\eta, \lambda)= \lim_{n\to\infty} i(Y_n,Z_n) = 0.$$

		We conclude that 
		$i(\xi,\lambda)+ i(\eta,\lambda)=0,$
		which contradicts the condition that $\xi$ and $\eta$ fill up the surface. 
	\end{proof}

		We have a quantitative version of Lemma \ref{lem:compact}.

	\begin{lemma}\label{lem:quantitative}
		There is a constant $R>0$ depending on $\xi$ and $\eta$
		such that $d(X_0, Z_n) < R$ when $n$ is sufficiently large. 
	\end{lemma}
	
	\begin{proof}
		This follows by modifying the proof of Lemma \ref{lem:compact}.
		As we have shown, 
		
		\begin{eqnarray*}
			-d(X_0, Z_n) &=& \log i(X_n, Z_n) \\
			&=& \log i(Y_n,Z_n) - \log i(X_n,Y_n).
		\end{eqnarray*}
		
		We write the above equation as 
		
		\begin{eqnarray*}
			e^{-d(X_0, Z_n)}  &=& i(X_n, Z_n) \\
			&=&  \frac{i(Y_n,Z_n)}{i(X_n,Y_n)}.
		\end{eqnarray*}	
		
		Since $i(X_n, Y_n) \to i(\xi,\eta)$ and $i(\xi, \eta)>0$,
		there is a constant $C=C(\xi,\eta)$ such that 
		for $n$ sufficiently large, 
		\begin{eqnarray}\label{equ: intersection}
			C	e^{-d(X_0, Z_n)}  &\geq &  i(X_n, Z_n) + i(Y_n,Z_n).
		\end{eqnarray}		
		
		Denote by 
		$$\delta = \inf_{\lambda} \  i(\xi, \lambda) + i(\eta, \lambda),$$	
		where the infimum is taken over all $\lambda$ in the Gardiner-Masur boundary. Then $\delta>0$. See Lemma \ref{lem:positive} in the Appendix.

		Using \eqref{equ: intersection},  we have  $$ \liminf\limits_{n\to \infty}	e^{-d(X_0, Z_n)} \geq \delta/C >0.$$

		Now we take $\mathcal{K}$ to be the geodesic ball centered at $X_0$ and of radius 
		$R= 2 \log \frac{C}{\delta}$.  Then  for any sequences $X_n \to \xi$ and $Y_n \to \eta$, the points $Z_n\in \overline{X_n Y_n}$ lie in
		$\mathcal{K}$ as $n$ is sufficiently large.  
	\end{proof}

	\begin{proof}[Proof of Theorem \ref{thm:convergence}]
		
		It follows from Lemma \ref{lem:compact} that, up to a subsequence, 
		we may assume that $Z_n$ converges to some $Z$ in $\TS$ and 
		$\mathbf{G}_n \to \mathbf{G}$. 	
		We proceed to show that $\mathbf{G}$ (up to orientation) is an optimal geodesic for 
		$\xi$ and $\eta$. 
		
		We take $Z$ to be the basis-point of horofunction compactification. 
		Note that, for fix $t$, 
		the horofunction corresponding to $\xi$ satisfies 
		
		\begin{eqnarray*}
			\Psi_\xi(\mathbf{G}_n(t)) &=& \lim_{n\to\infty} d(X_n, \mathbf{G}_n(t))-
			d(X_n,Z_n) \\
			&=& -t.
		\end{eqnarray*}
		
		By taking the limit, we have 
		
		\begin{eqnarray*}
			\Psi_\xi(\mathbf{G}(t))= -t.
		\end{eqnarray*}
		
		The same argument shows that 
		
		\begin{eqnarray*}
			\Psi_\eta(\mathbf{G}(t))= t.
		\end{eqnarray*}
		
		As a result, $\mathbf{G}$ is an optimal geodesic for 
		$\xi$ and $\eta$. Since optimal geodesic is unique (by Theorem \ref{thm:unique}),  $\mathbf{G}$ is exactly the limit of the whole sequence $\mathbf{G}_n$. This finish the proof. 
		
	\end{proof}

As a direct corollary of Lemma \ref{lem:quantitative} or Theorem \ref{thm:convergence},  we show that Teichm\"uller geodesics satisfy the following ``sticky" property.

	\begin{corollary}\label{thm:sticky}
	Let $\mathbf{G}$ be a Teichm\"uller geodesic, and let $\xi$ and $\eta$ be the endpoints of $\mathbf{G}$  in the Gardiner-Masur boundary. There is a compact set 
	$\mathcal{K}\subset \TS$ such that, for any sequences $X_n \to \xi$ and $Y_n \to \eta$, the geodesic segments $\overline{X_n Y_n}$ connecting $X_n$ and $Y_n$ intersect $\mathcal{K}$ for sufficiently large $n$.
	\end{corollary}
	
	The original conjecture of 
Duchin and Fisher \cite[Conjecture 3]{DF} is that,
if $\alpha$ and $\beta$ are simple closed curves that fill up the surface, and $\mathbf{G}$ is a Teichm\"uller geodesic ending at $\alpha$ and $\beta$, then  $\mathbf{G}$ is \textit{sticky} in the topology of Thurston  compactification: there is a compact set 
$\mathcal{K}\subset \TS$ such that, for any sequences $X_n \to \alpha$ and $Y_n \to \beta$ in the Thurston compactification, the geodesic segments $\overline{X_n Y_n}$ connecting $X_n$ and $Y_n$ intersect $\mathcal{K}$ for sufficiently large $n$.

 It seems  that $X_n \to \alpha$ in the Thurston compactification
  	doesn't imply that $X_n \to \alpha$ in the Gardiner-Masur compactification. 
It would be interesting to find a counterexample.

\appendix

	\section{Lemmas on intersection number}\label{sec:appendix}

	We first give a proof of  Lemma \ref{prop_low-bound}.
For any $X$ and $Y$ in $\TS$, denote by $K(X,Y)$
the quasiconformal dilation of the Teichm\"uller extremal map between $X$ and $Y$.

	\begin{lemma}\label{lem:app1}
	Let $X_n \in \mathcal{T}(S)$ converge to a boundary point $\xi \in \partial_{\mathrm{GM}}\mathcal{T}(S)$. Assume that $q_n \in {Q^1(X_0)}$ is the initial quadratic differential associated to the Teichm\"uller map from $X_0$ to $X_n$. Assume that the vertical measured foliations ${F}_{v}(q_n)$ of $q_n$ converge to a measured foliation
	${F}$. Then 
	$
	i({F},\mu) \leq i(\xi,\mu)
	$
	for all $\mu \in \mathcal{MF}$.	
	\end{lemma}
	\begin{proof}
		Recall the following inequality, due to Minsky \cite[Lemma 5.1]{Minsky1993}: 
	$$i(\nu, \mu)\leq \sqrt{\operatorname{Ext}_{X}(\nu)}
	\sqrt{\operatorname{Ext}_{X}(\mu)}$$
	for any $\mu, \nu\in \MF$. Taking $\nu_n ={F}_{v}(q_n)$
	and note that $\operatorname{Ext}_{X_n}(\nu_n)= \operatorname{Ext}_{X_0}(\nu_n)/ K(X_0, X_n)$, we have
	
	$$i(\nu_n, \mu)\leq \sqrt{\operatorname{Ext}_{X_0}(\nu_n)} \cdot \frac{\sqrt{\operatorname{Ext}_{X_n}(\mu)}}{\sqrt{K(X_0, X_n)}} .
	$$
	
Since $\sqrt{\operatorname{Ext}_{X_0}(\nu_n)}=1 $, as $X_n \to \xi$, we have
$\nu_n \to F$ and 
$$\frac{\sqrt{\operatorname{Ext}_{X_n}(\mu)}}{\sqrt{K(X_0, X_n)}} \to i(\xi, \mu),$$
which shows that 	$i({F},\mu) \leq i(\xi,\mu)$.
	
	\end{proof}

	The next lemma is a generalization of Lemma \ref{lem:app1}. 
	
		\begin{lemma}\label{lem:app2}
		Let $X_n \in \mathcal{T}(S)$ converge to a boundary point $\xi \in \partial_{\mathrm{GM}}\mathcal{T}(S)$. Assume that $q_n \in {Q^1(X_0)}$ is the initial quadratic differential associated to the Teichm\"uller map from $X_0$ to $X_n$. Assume that the vertical measured foliations ${F}_{v}(q_n)$ of $q_n$ converge to a measured foliation
		${F}$. Then 
		$
		i({F},Y) \leq i(\xi,Y)
		$
		for all $Y\in \TS$.	
	\end{lemma}
	\begin{proof}
	By definition, 	for $\nu\in \MF$ and $Y\in \TS$, 
	$$	i(\nu,Y) =\frac{\sqrt{\operatorname{Ext}_{Y}(\nu)}}{\sqrt{K(X_0, Y)}}. $$	
	If follows from a fundamental property of quasiconformal mapping that
	$$
	\operatorname{Ext}_{Y}(\nu) \leq K(X_n, Y) \operatorname{Ext}_{X_n}(\nu).
	$$	
	
	As before, denote by $\nu_n = F_v(q_n)$. Combining the above two equations, we have 
	
	\begin{eqnarray*}
		i(\nu_n, Y ) &=&  \frac{\sqrt{\operatorname{Ext}_{Y}(\nu_n)}}{\sqrt{K(X_0, Y)}}\\
		 &\leq&  \frac{\sqrt{K(X_n, Y)}}{\sqrt{K(X_0, Y)}} \sqrt{\operatorname{Ext}_{X_n}(\nu_n)}  \\		
		  &=&  \frac{\sqrt{K(X_n, Y)}}{\sqrt{K(X_0, Y)}} \frac{1}{\sqrt{K(X_0,X_n)}}\\		
		  &=& i(X_n, Y).	 		
	\end{eqnarray*}	
By letting $X_n \to \xi$ and then $\nu_n \to F$, we have 

$$i(F, Y) \leq i(\xi, Y).$$		
		
	\end{proof}
	
	By continuity, Lemma \ref{lem:app2} implies 
	
		\begin{lemma}\label{lem:app3}
		Let $X_n \in \mathcal{T}(S)$ converge to a boundary point $\xi \in \partial_{\mathrm{GM}}\mathcal{T}(S)$. Assume that $q_n \in {Q^1(X_0)}$ is the initial quadratic differential associated to the Teichm\"uller map from $X_0$ to $X_n$. Assume that the vertical measured foliations ${F}_{v}(q_n)$ of $q_n$ converge to a measured foliation
		${F}$. Then 
		$
		i({F},\lambda) \leq i(\xi,\lambda)
		$
		for all $\lambda$ in the Gardiner-Masur boundary.	
	\end{lemma}

As a corollary of Lemma \ref{lem:app3}, we prove 

\begin{proposition}\label{lem:positive}
Let $\xi$ and $\eta$ be a pair of points in the Gardiner-Masur boundary 
that fill up the surface. Set 
$$\delta(\xi,\eta) = \inf_\lambda \  i(\xi,\lambda)+ i (\eta, \lambda)$$
where the infimum is taken over all $\lambda$ in the Gardiner-Masur boundary. Then 
$$\delta(\xi,\eta)>0.$$
\end{proposition}
\begin{proof}
		Since the Gardiner-Masur boundary is compact,  there is some boundary point $\lambda_0$ such that
	$$\delta(\xi,\eta) = i(\xi,\lambda_0) + i (\eta, \lambda_0).$$
	By Lemma \ref{lem:app3}, there is a measured foliation ${F}_0$, associated with $\lambda_0$, such that
	$i(\xi, \lambda_0) \geq i(\xi, F_0)$ and 
	$i(\eta, \lambda_0) \geq i(\eta, F_0)$. As a result, the condition that $\xi$ and $\eta$ fill up the surface implies $$\delta(\xi,\eta) \geq i(\xi,F_0) + i (\eta, F_0)>0.$$
\end{proof}

\bibliographystyle{plain}

\bibliography{bibliography}
	
\end{document}